\documentclass[10pt, a4paper]{amsart}

\pdfoutput=1

\usepackage[T1]{fontenc}

\usepackage{mathrsfs, amsmath, amsthm, amssymb, stmaryrd, enumerate, bbm, MnSymbol}%
\usepackage{tikz}%
\usetikzlibrary{knots}
\usepackage[linktocpage]{hyperref}%
\usepackage[all]{xy}%
\xyoption{2cell}%
\UseAllTwocells%

\newdir{t>}{{\UseTips\dir{>}}}
\newdir{ >}{{}*!/-5pt/@{>}}

\DeclareMathOperator{\id}{id}

\newcommand{\defl}{\mathrel{\mathop:}=}

\theoremstyle{plain}
\newtheorem{thm}{Theorem}[section]
\newtheorem*{thm*}{Theorem}
\newtheorem{prop}[thm]{Proposition}
\newtheorem{lemma}[thm]{Lemma}
\newtheorem{cor}[thm]{Corollary}

\theoremstyle{definition}

\newtheoremstyle{citing}{}{}{\itshape}{}{\bfseries}{.}{ }{\thmnote{#3}}
\theoremstyle{citing}

\newtheoremstyle{citingdfn}{}{}{}{}{\bfseries}{.}{ }{\thmnote{#3}}
\theoremstyle{citingdfn}

\newtheorem*{question*}{}

\numberwithin{equation}{section}


\author{Daniel Sch\"appi}
\thanks{This research was supported by the DFG grant: SFB 1085 ``Higher invariants.''}
\address{Fakult{\"a}t f{\"u}r Mathematik,
Universit{\"a}t Regensburg,
93040 Regensburg,
Germany}
\email{daniel.schaeppi@ur.de}

\title{A counterexample to the Hermite ring conjecture}

\begin{document}

\begin{abstract}

 We show that there exists a stably free module over a polynomial ring which is not extended from the ground ring. This provides a counterexample to the Hermite ring conjecture.
  
\end{abstract}

\maketitle

\section{introduction}

 Let $R$ be a commutative ring. Recall that a finitely generated projective $R$-module $P$ is called \emph{stably free} if there exist $n,m \in \mathbb{N}$ such that $P \oplus R^n \cong R^m$. The ring $R$ is called \emph{Hermite} if all stably free $R$-modules are free. The Hermite ring conjecture is the conjecture that the polynomial ring $R[t]$ is Hermite whenever $R$ is Hermite. This was posed as a question by Lam in the introduction to \cite[p.~XI]{LAM_OLD}.
 
 Swan asked if all stably free $R[t]$ modules for an arbitrary ring $R$ are necessarily extended from $R$ (see Question~(B) in \cite[p.~114]{SWAN}). These two statements turn out to be equivalent, see for example \cite[Proposition~V.3.4]{LAM}. In this article, we will show that these two questions have a negative answer.
 
 \begin{thm}\label{thm:counterexample_to_Swan_question}
 There exists a commutative ring $R$ and a stably free $R[t]$-module $P$ which is not extended from $R$.
 \end{thm}

 We prove this result by exploiting a connection between the Hermite ring conjecture and symplectic $K$-theory established in \cite[\S 8]{SCHAEPPI_SYMPLECTIC}. There it is shown that there exists a ring $R$ such that the Hermite ring conjecture fails if the forgetful homomorphism
 \[
 K_1 \mathrm{Sp}(R) \rightarrow K_1(R)
 \]
 from symplectic $K$-theory to ordinary $K$-theory is not injective (see \cite[Theorem~1.3]{SCHAEPPI_SYMPLECTIC}). Here $K_1 \mathrm{Sp}(R)$ is the quotient $\mathrm{
 Sp}(R) \slash \mathrm{Ep}(R)$ of the infinite symplectic group modulo the elementary symplectic group $\mathrm{Ep}(R)$ (see for example \cite[\S 1]{KOPEIKO} for a definition). Thus our goal is to construct a symplectic matrix over $R$ which is stably elementary when considered as a matrix in $\mathrm{SL}_n$, but does not lie in $\mathrm{Ep}(R)$. Such a matrix can then be used in a Milnor patching diagram to obtain the desired counterexample.
 
 The hard part of the argument is to show that the matrix in question represents a nontrivial element of $K_1 \mathrm{Sp}(R)$. To prove this, we adapt an argument of Gubeladze \cite{GUBELADZE}. In \cite[Proposition~5.1]{GUBELADZE}, Gubeladze showed that certain elements of $K_1(R)$ are nontrivial. We use the same scaffolding in our argument, though we replace the use of the Steinberg group for $\mathrm{GL}(R)$ by the symplectic Steinberg group. This argument relies on two fundamental ingredients: we need to establish a homotopy invariance property for $K_2 \mathrm{Sp}$ over fields, and we need to show that a specific element in the symplectic Steinberg group of a certain quotient of $k[s,t]$ is nonzero. (Gubeladze uses the quotient $k[s,t] \slash (st)$; we will use $k[s,t] \slash (s^2,st,t^2)$ instead.)
 
 The first of these two facts can be established using the recent work of Calm{\`e}s, Dotto, Harpaz, Hebestreit, Land, Moi, Nardin, Nikolaus, and Steimle \cite{CALMES_ET_AL_I, CALMES_ET_AL_II, CALMES_ET_AL_III, HEBESTREIT_STEIMLE, CALMES_HARPAZ_NARDIN}, see \S \ref{section:homotopy_invariance}.
 
 To establish the second fact, we construct a homomorphism
 \[
 \rho \colon K_2 \mathrm{Sp} \bigl(k[s,t]\slash (s^2,st,t^2)\bigr) \rightarrow W\bigl(k[s,t]\slash (s^2,st,t^2)\bigr)
 \]
 to the Witt ring of symmetric bilinear spaces and we show that the element in question is sent to a nontrivial element in the Witt ring. The homomorphism is constructed using a result of van der Kallen \cite{VdK} and the nontriviality in the Witt ring is shown with the aid of a theorem of Knebusch \cite{KNEBUSCH}.
 
 \subsection*{Acknowledgments}
 I am very grateful to Markus Land for proving the necessary homotopy invariance property of $K_2 \mathrm{Sp}$ over fields.

 \section{The matrix}\label{section:matrix}

 Let $k$ be a field of characteristic $2$ and let $u \in k^{\times}$. Let $R \defl k[a,x,y] \slash a^2+xy$. The goal of this section is to give a proof strategy to show that the matrix
 \[
 M_0(u)= \begin{pmatrix}
 1+ (1+u)a^2 & (1+u)(1+a)y \\ (1+u^{-1})(1+a)x & 1+(1+u^{-1})a^2
 \end{pmatrix} \in \mathrm{SL}_2(R)
 \]
 gives a nontrivial element of the kernel of $K_1 \mathrm{Sp}(R) \rightarrow K_1(R)$ if $u$ is not a square. The missing ingredients in this proof strategy will be supplied in the later sections. That $M_0(u)$ lies in the kernel of this homomorphism follows from a computation with Mennicke symbols as defined in \cite[\S VI.3]{LAM}.
 
 \begin{lemma}\label{lemma:matrix_1_stably_elementary}
 The matrix
 \[
  \begin{pmatrix}
  M_0(u) & 0 \\ 0 & 1
  \end{pmatrix} \in \mathrm{SL}_3(R)
 \]
 is elementary. 
 \end{lemma}

\begin{proof}
 We use the computation rules for Mennicke symbols proved in \cite[Proposition~VI.3.4]{LAM}. It suffices to show that the Mennicke symbol
 \[
  [1+(1+u)a^2,(1+u)(1+a)y]
 \]
 is trivial. Clearly $[1+(1+u)a^2,(1+u)]$ is trivial and $[1+(1+u)a^2,y]$ is trivial since $a^2=xy$ in $R$. Since Mennicke symbols are multiplicative, it only remains to show that $[1+(1+u)a^2, 1+a]$ is trivial. We have
 \begin{align*}
 [1+(1+u)a^2,1+a] &=[1=(1+u)a^2+(1+u)(1+a^2),1+a] \\
 &=[1+1+u, 1+a]\\
 &=[u,1+a]=1 \smash{\rlap{,}}
 \end{align*}
 so the matrix in question is indeed elementary.
\end{proof}

 Showing that $M_0(u)$ represents a nontrivial element of $K_1 \mathrm{Sp}(R)$ is more invovled. In \cite{GUBELADZE}, Gubeladze has shown that a certain matrix is nontrivial in $K_1(R)$. We adapt his proof to the case of symplectic $K$-theory $K_1 \mathrm{Sp}(R)$. For this, it is convenient to consider the matrix
 \[
 M(u)=\begin{pmatrix}
 1+ (1+u)a^2 & (1+u)(1+a)y \\ (1+u^{-1})(1+a)x & 1+(1+u^{-1})a^2
 \end{pmatrix}
 \begin{pmatrix}
 1 & 0 \\ (1+u^{-1})x & 1
 \end{pmatrix}
 \]
 instead, which represents the same element as $M_0(u)$ in $K_1\mathrm{Sp}(R)$.
 
 We consider the homomorphism $\psi \colon R \rightarrow k[s,t]$ to the polynomial ring in two variables given by $\psi(x)=t$, $\psi(a)=st$, and $\psi(y)=s^2t$. (This identifies $R$ with the subring $k[t,st,s^2t]$ of $k[s,t]$.) Under this homomorphism, we have
 \[
 \psi \bigl(M_0(u)\bigr)=\begin{pmatrix}
 1+ (1+u)s^2t^2 & (1+u)(1+a)s^2t \\ (1+u^{-1})(1+a)t & 1+(1+u^{-1})s^2t^2
 \end{pmatrix} \smash{\rlap{,}}
 \]
 which happens to be an elementary matrix, we have
 \[
 \psi\bigl(M(u)\bigr)=
 \begin{pmatrix}
 1 & us \\ 0 & 1
 \end{pmatrix}
 \begin{pmatrix}
 1 & 0 \\ u^{-1}t & 1
 \end{pmatrix}
 \begin{pmatrix}
 1 & (1+u)s \\ 0 & 1
 \end{pmatrix}
 \begin{pmatrix}
 1 & 0 \\ t & 1
 \end{pmatrix}
 \begin{pmatrix}
 1 & s \\ 0 & 1
 \end{pmatrix}
 \begin{pmatrix}
 1 & 0 \\ (1+u^{-1})t & 1
 \end{pmatrix}
 \]
 in $\mathrm{SL}_2(k[s,t])$.

 Gubeladze's proof uses Steinberg groups, so we will similarly work in the symplectic Steinberg group $\mathrm{StSp}(A)$ of a commutative ring $A$. This is the directed colimit of the Steinberg groups $\mathrm{StSp}_{2r}$ of the symplectic group $\mathrm{Sp}_{2r}$. Here $\mathrm{Sp}_{2r}(A)$ is the group of $2r \times 2r$ matrices $M$ which leave the alternating form
 \[
 \chi_{2r}= \begin{pmatrix}
 0 & 1 \\ -1 & 0 \\ && \ddots \\ &&& 0 & 1\\ &&& -1 & 0
 \end{pmatrix}
 \]
 invariant, that is, $M^t \chi_{2r} M=\chi_{2r}$. The symplectic Steinberg group $\mathrm{StSp}_{2r}(A)$ has generators
 \[
 x_{\beta}(\xi) \quad \xi \in A, \; \beta \in \Phi_r \smash{\rlap{,}}
 \]
 where $\Phi_r$ denotes the root system of $\mathrm{Sp}_{2r}$. These generators are subject to two relations: for all $\xi, \xi^{\prime} \in A$ and fixed $\beta \in \Phi_r$ the equation
 \[
 x_{\beta}(\xi + \xi^{\prime})=x_{\beta}(\xi)x_{\beta}(\xi^{\prime})
 \]
 holds, and a second relation involving commutators for distinct $\beta, \beta^{\prime} \in \Phi_r$ (see for example \cite[\S 3]{STEIN_COVERING} for a precise definition of the Steinberg group of a Chevalley--Demazure group scheme). We will only need the first relation in our computations.
 
 Each root $\beta \in \Phi_r$ has an associated \emph{root unipotent} $\xi \mapsto e_{\beta}(\xi) \in \mathrm{Sp}_{2r}$. Sending $x_{\beta}(\xi)$ to $e_{\beta}(\xi)$ defines a surjective group homomorphism $\mathrm{StSp}_{2r}(A) \rightarrow \mathrm{Ep}_{2r}(A)$. We orient the root system so that the unique simple long root $\alpha$ of $\mathrm{Sp}_{2r}$ has root unipotent $\xi \mapsto e_{\alpha}(\xi)$ given by $e_{12}(\xi)$. That this is possible follows for example from the explicit description of the root unipotents given in \cite{BAK_STEPANOV}, with one caveat. We have to transform the basis of $A^{2r}$ used in \cite{BAK_STEPANOV} indexed by the ordered set $I=(1,2, \ldots, r, -r, \ldots, -1)$ to the one indexed by $\bigl(r,-r,(r-1),-(r-1),\ldots, 1,-1 \bigr)$. This has the effect of transforming the alternating form $\bigl( \begin{smallmatrix} 0 & J \\ -J & 0 \end{smallmatrix} \bigr)$ of \cite{BAK_STEPANOV} into $\chi_{2r}$ and sending the root unipotent $e_{2 \varepsilon_r}(\xi)$ to $e_{12}(\xi)$.
 
 The upshot of this is that the element
 \[
 z \defl x_{\alpha}(us)x_{-\alpha}(u^{-1}t)x_{\alpha}\bigl((1+u)s\bigr)x_{-\alpha}(t) x_{\alpha}(s)x_{-\alpha} \bigl((1+u^{-1})t\bigr) \in \mathrm{StSp}(k[s,t])
 \]
 is sent to $\psi\bigl(M(u)\bigr)$ under the canonical map $\mathrm{StSp}(k[s,t]) \rightarrow \mathrm{Ep}(k[s,t])$, where $\alpha$ denotes the unique simple long root.

 The following proposition is an adaptation of Step~1 of the proof of \cite[Proposition~5.1]{GUBELADZE} to the case of symplectic Steinberg groups and the matrix $M(u)$. Recall that the kernel of the natural surjection $\mathrm{StSp}(A) \rightarrow \mathrm{Ep}(A)$ is the second symplectic $K$-group $K_2 \mathrm{Sp}(A)$.
 
 \begin{prop} \label{prop:gubeladze} 
 Assume the following two facts:
\begin{enumerate}
\item[(1)] (Homotopy invariance) The inclusion $k \rightarrow k[s,t]$ induces an isomorphism $K_2 \mathrm{Sp}(k) \rightarrow K_2 \mathrm{Sp}(k[s,t])$;
\item[(2)] (Nontriviality in the quotient) The image $\bar{z} \in \mathrm{StSp}\bigl(k[s,t]\slash (s^2,st,t^2)\bigr)$ of $z$ under the canonical projection is nontrivial if $u \in k \setminus k^2$.
\end{enumerate}

Then  $[M(u)] \in K_1 \mathrm{Sp}(R)$ is nontrivial whenever $u$ is not a square. 
 \end{prop}

\begin{proof}
 Assume that $u$ is not a square but $[M(u)]$ is trivial, that is, $M(u) \in \mathrm{Ep}(R)$. Then there exists an element $w \in \mathrm{StSp}(R)$
 whose image under the surjective homomorphism
 \[
 \mathrm{StSp}(R) \rightarrow \mathrm{Ep}(R)
 \]
 is $M(u)$.
 
 It follows that $\psi_{\ast} w \in \mathrm{StSp}(k[s,t])$ has image $\psi\bigl(M(u)\bigr)$ in $\mathrm{Ep}(k[s,t])$. Since $z$ has the same image in $\mathrm{Ep}(k[s,t])$, the product $v \defl \psi_{\ast} w \cdot z^{-1}$ lies in the kernel $K_2 \mathrm{Sp}(k[s,t])$ of the map
 \[
 \mathrm{StSp}(k[s,t]) \rightarrow \mathrm{Ep}(k[s,t]) \smash{\rlap{.}}
 \]
 From Assumption~(1) we conclude that $v$ lies in the image of the isomorphism $K_2 \mathrm{Sp}(k) \rightarrow K_2 \mathrm{Sp}(k[s,t])$. More precisely, if we write $\iota \colon k \rightarrow k[s,t]$ for the canonical inclusion, then there exists a (unique) $v_0 \in K_2 \mathrm{Sp}(k)$ such that $v=\iota_{\ast} v_0$. But $\iota \colon k \rightarrow k[s,t]$ factors through $\psi \colon R \rightarrow k[s,t]$. Indeed, if we write $\iota^{\prime} \colon k \rightarrow R$ for the canonical inclusion, we find that $\iota=\psi \iota^{\prime}$. Thus we can modify $w$ to $w^{\prime}=(\iota^{\prime}_{\ast} v_0)^{-1}w$ to obtain an element $w^{\prime} \in \mathrm{StSp}(R)$ with
 \[
 (\psi_{\ast} w^{\prime})z^{-1}=(\psi_{\ast} \iota^{\prime}_{\ast} v_0)^{-1}\psi_{\ast} w z^{-1}=v^{-1}v=1 \smash{\rlap{,}}
 \]
 that is, $\psi_{\ast} w^{\prime}=z$.
 
 Now let $\varphi \colon k[x] \rightarrow k[s,t] \slash (s^2,st,t^2)$ be the unique $k$-homomorphism satisfying $\varphi(x)=t$. Since the projection $\pi \colon k[s,t] \rightarrow k[s,t] \slash (s^2,st,t^2)$ sends $st$ to $0$, it follows that the composite
 \[
 \xymatrix{R \ar[r]^-{\psi} & k[s,t] \ar[r]^-{\pi} & k[s,t]\slash (s^2,st,t^2)}
 \]
 factors through $\varphi$. The image $w^{\prime\prime}$ of $w^{\prime}$ in $\mathrm{StSp}(k[x])$ thus gives an element with $\varphi_{\ast} w^{\prime \prime}=\pi_{\ast}z=\bar{z}$.
 
 Finally, let $\zeta \colon k[s,t] \slash (s^2,st,t^2) \rightarrow k[s,t] \slash (s^2,st,t^2)$ be the endomorphism given by $\zeta(s)=0$, $\zeta(t)=t$. Then $\zeta \varphi=\varphi$, so from the commutative triangle
 \[
 \xymatrix{& \mathrm{StSp}(k[x]) \ar[rd]^{\varphi_{\ast}} \ar[ld]_{\varphi_{\ast}} \\
 \mathrm{StSp}\bigl(k[s,t]\slash(s^2,st,t^2)\bigr) \ar[rr]_{\zeta_{\ast}} && \mathrm{StSp}\bigl(k[s,t]\slash(s^2,st,t^2)\bigr) }
 \]
 we conclude that $\zeta_{\ast}(\pi_{\ast} z)=\zeta_{\ast} \varphi_{\ast}w^{\prime\prime}=\varphi_{\ast} w^{\prime \prime}=\pi_{\ast} z$, that is, $\zeta_{\ast} \bar{z}=\bar{z}$.
 
 Since we have
\begin{align*}
\zeta_{\ast} \bar{z} &= x_{-\alpha}(u^{-1} \bar{t}) x_{-\alpha}(\bar{t}) x_{-\alpha}\bigl((1+u^{-1})\bar{t}\bigr)\\
&=x_{-\alpha}\bigl(u^{-1}\bar{t}+\bar{t}+(1+u^{-1})\bar{t}\bigr)\\
&=1 \smash{\rlap{,}}
\end{align*} 
 this contradicts Assumption~(2) that $\bar{z}$ is nontrivial.
\end{proof}

\section{Homotopy invariance}\label{section:homotopy_invariance}

 The following argument was supplied by Markus Land. We use the notation of \cite{CALMES_ET_AL_III} and we write $\mathrm{GW}^{-ge}(R)$ for $\mathrm{GW}^{ge}(R;-R)$ and similarly for other flavours of Grothendieck--Witt and $L$-theory. For the definition of the genuine even Poincar{\'e} structure, see \cite[Proposition~4.2.22]{CALMES_ET_AL_I} and \cite[Notation~4.2.23]{CALMES_ET_AL_I}.
  
 \begin{prop}[Land]\label{prop:land}
 For any field $k$, the canonical morphism
 \[
 \mathrm{GW}_2^{-ge}(k) \rightarrow \mathrm{GW}_2^{-ge}(k[s,t])
 \]
 is an isomorphism.
 \end{prop} 

 \begin{proof}
 Let $A=k[s,t]$. From the fundamental fiber sequence \cite[Corollary~4.4.14]{CALMES_ET_AL_II}, we get the diagram
 \[
 \xymatrix@C=19pt{L_3^{-ge}(k) \ar[r] \ar[d] & \pi_2\bigl(K(k)_{hC_2}\bigr) \ar[r] \ar[d] & \mathrm{GW}_2^{-ge}(k) \ar[r] \ar[d] & L_2^{-ge}(k) \ar[r] \ar[d] & \pi_1\bigl(K(k)_{hC_2}\bigr) \ar[d] \\
 L_3^{-ge}(A) \ar[r] & \pi_2\bigl(K(A)_{hC_2}\bigr) \ar[r] & \mathrm{GW}_2^{-ge}(A) \ar[r]  & L_2^{-ge}(A) \ar[r] & \pi_1\bigl(K(A)_{hC_2}\bigr)}
 \]
 with exact rows. The second and fourth vertical homomorphism are isomorphisms since $K$-theory is $\mathbb{A}^{1}$-invariant for regular rings (such as fields). From the five lemma it follows that it suffices to check that
 \[
 L^{-ge}_n(k) \rightarrow L^{-ge}_n(k[s,t])
 \]
 is an isomorphism for $n=2,3$. Since $L^{gs}(R;R) \simeq \Omega^2 L^{ge}(R;-R)$, we have isomorphisms $L_n^{-ge}(R) \cong L^{gs}_{n-2}(R)$ with genuine symmetric $L$-groups for all rings $R$ (see \cite[Corollary~R.10]{CALMES_ET_AL_III}).
 
 Since $k[s,t]$ is regular of global dimension $2$, the top horizontal map in the square
 \[
 \xymatrix{L_n^{gs}(k[s,t]) \ar[r]  \ar[d] & L^s(k[s,t]) \ar[d] \\ L_n^{gs}(k) \ar[r] & L_n^s(k)}
 \]
 is an isomorphism if $n=1$ and injective if $n=0$ (see \cite[Corollary~1.3.10]{CALMES_ET_AL_III}). From \cite{CALMES_HARPAZ_NARDIN} we know that the right vertical homomorphism above is an isomorphism. It follows that the split epimorphism $L_n^{gs}(k[s,t]) \rightarrow L_n^{gs}(k)$ is injective for $n=0,1$, hence an isomorphism. Thus its section $L_n^{gs}(k) \rightarrow L_n^{gs}(k[s,t])$ is an isomorphism, too, which concludes the proof.
 \end{proof}

\begin{cor}\label{cor:homotopy_invariance}
For any field $k$, the homomorphism
\[
K_2 \mathrm{Sp}(k) \rightarrow K_2 \mathrm{Sp}(k[s,t])
\]
induced by the natural inclusion is an isomorphism.
\end{cor}

\begin{proof}
 It follows from results of Hebestreit and Steimle that $\mathrm{GW}^{ge}(R;-R)$ is equivalent to the group completion of $\mathrm{Iso}\bigl(\mathbf{A}(R)\bigr)$, the groupoid of finitely generated projective $R$-modules equipped with an alternating form (see \cite[Corollary~8.1.8]{HEBESTREIT_STEIMLE} and the proof of \cite[Proposition~8.20]{SCHAEPPI_SYMPLECTIC} for details). The second homotopy group of this group completion is naturally isomorphic to $K_2 \mathrm{Sp}(R)$ (this follows from \cite[Proposition~8.22]{SCHAEPPI_SYMPLECTIC}, applied to the basic object given by the hyperbolic plane), so the claim follows from Proposition~\ref{prop:land}.
\end{proof}

\section{Nontriviality in the quotient}\label{section:nontriviality}

 Throughout this section we fix a field $k$ of characteristic $2$. We first express the element
 \[
 z = x_{\alpha}(us)x_{-\alpha}(u^{-1}t)x_{\alpha}\bigl((1+u)s\bigr)x_{-\alpha}(t) x_{\alpha}(s)x_{-\alpha} \bigl((1+u^{-1})t\bigr) \in \mathrm{StSp}(k[s,t])
 \]
 in terms of commutators.

\begin{lemma}\label{lemma:z_as_commutators}
 The equation
 \[
 z=[x_{\alpha}(us),x_{-\alpha}(u^{-1}t)][x_{-\alpha}(u^{-1}t),x_{\alpha}(s)][x_{\alpha}(s),x_{-\alpha}\bigl((1+u^{-1})t\bigr)]
 \]
 holds in $\mathrm{StSp}(k[s,t])$.
\end{lemma}

\begin{proof}
 From the relation $x_{\beta}(a+b)=x_{\beta}(a) x_{\beta}(b)$ in the Steinberg group and the fact that $k$ has characteristic $2$ it follows that the generators $x_{\beta}(a)$ all have order $2$. Moreover, given two elements $g$ and $h$ of order $2$ in an arbitrary group we have $gh=hgghgh=hg[g,h]$. To keep the formulas legible, we use the abbreviations $g_3=[x_{\alpha}(s),x_{-\alpha}\bigl((1+u^{-1})t\bigr)]$, $g_2=[x_{-\alpha}(u^{-1}t),x_{\alpha}(s)]$, and $g_1=[x_{\alpha}(us),x_{-\alpha}(u^{-1}t)]$ for the three commutators in question, so we have
 \begin{align*}
 z&=x_{\alpha}(us) x_{-\alpha}(u^{-1}t)x_{\alpha}\bigl((1+u)s\bigr) \underbrace{x_{-\alpha}(t) x_{-\alpha}\bigl((1+u^{-1})t\bigr)}_{=x_{-\alpha}(u^{-1}t)}x_{\alpha}(s) g_3 \\
 &=x_{\alpha}(us) x_{-\alpha}(u^{-1}t)\underbrace{x_{\alpha} \bigl((1+u)s\bigr) x_{\alpha}(s)}_{=x_{\alpha}(us)} x_{-\alpha}(u^{-1}t)g_2 g_3\\
 &=x_{\alpha}(us) \underbrace{x_{-\alpha}(u^{-1}t) x_{-\alpha}(u^{-1}t)}_{=1} x_{\alpha}(us) g_1 g_2 g_3\\
 &=g_1 g_2 g_3 \smash{\rlap{,}}  
 \end{align*}
 as claimed.
\end{proof}

 The following result of Stein is helpful for expressing these commutators in the quotient ring $k[s,t] \slash (s^2, st, t^2)$ in terms of Steinberg symbols. We refer the reader to \cite[\S 3]{VdK} for the definition and basic properties of Steinberg symbols. We note that the definition of \cite[\S 1]{STEIN_STABILITY} differs from the one in \cite{VdK} by a sign: the Steinberg symbol $\{u,v\}_{\beta}$ as defined in \cite[\S 1]{STEIN_STABILITY} is equal to the Steinberg symbol $\{u,v\}_{-\beta}$ with the convention of \cite[\S 3]{VdK}.
 
 \begin{lemma}[Stein]\label{lemma:Stein_commutator_formula}
  If $p,q \in A$ are elements with $pq=0$ and $1+p,1+q \in R^{\times}$, then for all roots $\beta$ we have
  \[
  [x_{-\beta}(p),x_{\beta}(q)]=\{1+p,1+q\}_{\beta}
  \]
  in $\mathrm{StSp}(A)$. In particular, this commutator is a central element of $\mathrm{StSp}(A)$.
 \end{lemma}
 
 \begin{proof}
 The first part is shown in \cite[Corollary~2.9]{STEIN_STABILITY} for arbitrary Chevalley--Demazure group schemes. More precisely, the last line of the proof of \cite[Corollary~2.9]{STEIN_STABILITY} yields the equality
 \[
  [x_{-\beta}(p),x_{\beta}(q)]=\{1+p,1+q\}_{-\beta} \smash{\rlap{,}}
 \]
 which gives the desired result by the above remark. That $\{1+p,1+q\}_{\beta}$ is a central element is for example shown in \cite[Proposition~1.3.(a)]{STEIN_STABILITY}.
 \end{proof}

\begin{lemma}\label{lemma:commutator_product}
 In the ring $k[s,t] \slash (s^2, st, t^2)$, the equality
 \[
 [x_{\alpha}(\bar{s}),x_{-\alpha}\bigl((1+u^{-1})\bar{t}\bigr)]=[x_{\alpha}(\bar{s}),x_{-\alpha}(u^{-1}\bar{t})][x_{\alpha}(\bar{s}),x_{-\alpha}(\bar{t})]
 \]
 holds.
\end{lemma}

\begin{proof}
Since $x_{-\alpha}\bigl((1+u^{-1})\bar{t}\bigr)=x_{-\alpha}(u^{-1}\bar{t}) x_{-\alpha}(\bar{t})$, this follows from the commutator formula $[a,bc]=[a,b]\cdot b[a,c]b^{-1}$ and the centrality of $[x_{\alpha}(\bar{s}),x_{-\alpha}(\bar{t})]$ (see Lemma~\ref{lemma:Stein_commutator_formula}).
\end{proof}

\begin{lemma}\label{lemma:final_form}
 The image of $z$ in $\mathrm{StSp}\bigl(k[s,t] \slash (s^2,st,t^2) \bigr)$ lies in $K_2 \mathrm{Sp}$ and is given by
 \[
 \bar{z}=\{1+u^{-1} \bar{t},1+u\bar{s}\}_{\alpha} \{1+\bar{t},1+\bar{s}\}_{\alpha}^{-1} \smash{\rlap{.}}
 \]
\end{lemma}

\begin{proof}
 That $\bar{z}$ lies in $K_2 \mathrm{Sp}$ follows from the computation of the image of $z$ in $\mathrm{Ep}(k[s,t])$ after reduction modulo $(s^2,st,t^2)$ (since this image is given by $\psi \bigl(M(u)\bigr)$, which is equal to the identity already modulo $st$).

 Combining the computation in $k[s,t]$ (see Lemma~\ref{lemma:z_as_commutators}) with the product formula of Lemma~\ref{lemma:commutator_product}, we find that
 \[
 \bar{z}=[x_{\alpha}(u\bar{s}),x_{-\alpha}(u^{-1}\bar{t})][x_{\alpha}(\bar{s}),x_{-\alpha}(\bar{t})]
 \]
 holds, which is equal to
 \[
 \{1+u\bar{s},1+u^{-1} \bar{t}\}_{-\alpha} \cdot  \{1+\bar{s},1+ \bar{t}\}_{-\alpha}
 \]
 by Lemma~\ref{lemma:Stein_commutator_formula}. By \cite[3.2.(j)]{VdK}, we have
 \[
 \bar{z}= \{1+u^{-1}\bar{t}, 1+u\bar{s}\}_{\alpha}^{-1} \cdot  \{1+ \bar{t},1+\bar{s}\}_{\alpha}^{-1}
 \]
 in $\mathrm{StSp}\bigl(k[s,t] \slash (s^2,st, t^2)\bigr)$.
 
 Since $1+u\bar{s}=-(1+u\bar{s})^{-1}$, we have by \cite[3.2.(f4)]{VdK} and \cite[3.2.(f7)]{VdK} the equality
 \[
 \{1+u^{-1}\bar{t}, 1+u\bar{s}\}_{\alpha} \{1+u^{-1}\bar{t}, 1+u\bar{s}\}_{\alpha}=\{1+u^{-1}\bar{t},1\}_{\alpha}=1 \smash{\rlap{,}}
 \]
 so $\{1+u^{-1}\bar{t}, 1+u\bar{s}\}_{\alpha}^{-1}=\{1+u^{-1}\bar{t}, 1+u\bar{s}\}_{\alpha}$ holds, which concludes the proof.
\end{proof}

Similarly one can show that $\{1+\bar{t},1+\bar{t}\}_{\alpha}$ has order $2$, so we could rewrite the formula of the above lemma so that no inverses appear. The form given in the lemma will however be more convenient for us. In order to show that $\bar{z}$ is a nontrivial element of $K_2 \mathrm{Sp}\bigl(k[s,t]\slash(s^2,st,t^2)\bigr)$, we will use the following theorem of van der Kallen.
 
 \begin{thm}[van der Kallen]\label{thm:presentation_of_K2Sp}
 Let $A$ be an $U$-irreducible ring (see \cite[Definition~1.3]{VdK}), for example, a semilocal ring whose residue fields are all infinite (see \cite[Example~1.5.(d)]{VdK}). Then $K_2 \mathrm{Sp}(A)$ is isomorphic to the group with generators $\{u,v\}_{\alpha}$ for $u, v \in A^{\times}$, subject to the following relations for all $x,y,z,w \in A^{\times}$ such that $(1-w) \in A^{\times}$:
 \begin{enumerate}
 \item[(a)] $\{x,y\}_{\alpha} \{xy,z\}_{\alpha}=\{x,yz\}_{\alpha} \{y,z\}_{\alpha}$;
 \item[(b)] $\{1,1\}_{\alpha}=1$;
 \item[(c)] $\{x,y\}_{\alpha}=\{x^{-1},y^{-1}\}_{\alpha}$;
 \item[(e)] $\{w,y\}_{\alpha}=\{w,(1-w)y\}_{\alpha}$.
\end{enumerate}  
 \end{thm}
 
 \begin{proof}
 This follows from \cite[Theorem~3.4]{VdK} and \cite[Remark~3.5]{VdK}, which states that the omitted relation (d) follows from the others for all $U$-irreducible rings $A$.
 \end{proof}
 
 Suslin has constructed a homomorphism from $K_2 \mathrm{Sp}(F)$ to the Witt ring $W(F)$ of a field $F$ (see \cite[\S 6]{SUSLIN}). With the above theorem, we can extend this from fields to $U$-irreducible rings. Namely, we need to describe elements in the Witt ring which satisfy the above four relations. It turns out that the two-fold Pfister forms do satisfy these relations over a general ring. A \emph{symmetric bilinear space} is a finitely generated projective module equipped with a symmetric bilinear form. If the underlying module is free, we will usually denote the space by the corresponding matrix. We write $\langle a_1, \ldots, a_n \rangle$ for a diagonal bilinear form with the $a_i$ on the diagonal. The two-fold \emph{Pfister form} is defined to be
 \[
 \llangle a,b \rrangle \defl \langle 1,-a \rangle \otimes \langle 1,-b \rangle \cong \langle 1,-a,-b,ab \rangle
 \]
 for $a, b \in A^{\times}$. We use Knebusch's definition \cite[\S 3.5]{KNEBUSCH} of the Witt ring $W(A)$ of a commutative ring $A$ and we also use his notation $A(\lambda,\mu)$ for the symmetric bilinear form with the matrix $\bigl( \begin{smallmatrix} \lambda & 1 \\ 1 & \mu \end{smallmatrix} \bigr)$.
 
 \begin{prop}\label{prop:relations_in_Witt_ring}
 Let $A$ be a commutative ring. Then the following relations hold in $W(A)$:
 \begin{enumerate}
 \item[(a)] $\llangle a, b \rrangle \llangle ab,c \rrangle=\llangle a, bc \rrangle \llangle b, c \rrangle$ for all $a,b,c \in A^{\times}$;
 \item[(b)] $\llangle 1,1 \rrangle=1$;
 \item[(c)] $\llangle a,b \rrangle=\llangle a^{-1}, b^{-1} \rrangle$ for all $a,b \in A^{\times}$;
 \item[(e)] $\llangle a, b \rrangle=\llangle a, (1-a)b \rrangle$ for all $a,b \in A^{\times}$ such that $1-a \in A^{\times}$.
 \end{enumerate}
 \end{prop} 
 
 \begin{proof}
 For the first claim, note that the left hand side is
 \[
 \langle 1,-a,-b,ab \rangle+\langle 1,-ab,-c,abc \rangle = \langle 1,1,-a,-b,-c,abc \rangle + \langle ab,-ab\rangle
 \]
 and the right hand side is
 \[
 \langle 1,-a,-bc,abc \rangle + \langle 1,-b,-c,bc \rangle = \langle 1,1,-a,-b,-c,abc \rangle + \langle bc,-bc\rangle \smash{\rlap{,}}
 \]
 so the claim follows from the fact that $\langle \lambda,-\lambda \rangle=0$ in $W(A)$ (see \cite[Folgerung~3.1.4]{KNEBUSCH}).
 
 The second claim also follows immediately from this fact.
 
 To see $(c)$, note that we have to show that
 \[
 \langle 1,-a,-b,ab \rangle= \langle 1, -a^{-1},-b^{-1}, (ab)^{-1} \rangle
 \]
 or equivalently that
\[
\langle a,-a^{-1} \rangle + \langle b ,-b^{-1} \rangle + \langle -ab,(ab)^{-1} \rangle=0
\]
holds.

 For $x \in A^{\times}$, the bilinear form $\langle x, -x^{-1} \rangle$ has matrix $A(x,0)$ with respect to the basis $\bigl(\begin{smallmatrix} 1 \\ 0 \end{smallmatrix} \bigr)$, $\bigl(\begin{smallmatrix} x^{-1} \\ 1 \end{smallmatrix} \bigr)$, so the left hand side above is metabolic, hence equal to $0$ in $W(A)$.
 
 For the final claim, note that
 \begin{align*}
 \llangle a, b \rrangle - \llangle a, (1-a)b \rrangle &= \langle 1,-a,-b,ab \rangle- \langle 1, -a, -(1-a)b,a(1-a)b \rangle \\
 &=\langle -b,ab,(1-a)b,-a(1-a)b \rangle
 \end{align*}
 holds. Since $1-a$ is a unit, the four vectors
 \[
 v_1 \defl \bigl((1-a)b\bigr)^{-1} e_3, \quad 
 v_2 \defl \begin{pmatrix}
 1 \\ 1 \\ 1 \\ 0
 \end{pmatrix}, \quad
 v_3 \defl \bigl(-a(1-a)b\bigr)^{-1} e_4, \quad
 v_4 \defl \begin{pmatrix}
 a \\ 1 \\ 0 \\ 1
 \end{pmatrix}
 \]
 form a basis of $A^4$. With respect to this basis, the bilinear form above has the matrix
 \[
 \begin{pmatrix}
 \lambda & 1 & 0 & 0 \\ 1 & 0 & 0 & 0 \\ 0 & 0 & \mu & 1 \\ 0 & 0 & 1 & 0
 \end{pmatrix}=A(\lambda, 0) \perp A(\mu, 0)
 \]
 for some $\lambda, \mu \in A^{\times}$, hence it is metabolic, so equal to $0$ in $W(A)$. 
 \end{proof}
 
 \begin{cor}\label{cor:homomorphism_to_Witt_ring}
  If $A$ is $U$-irreducible in the sense of \cite[Definition~1.3]{VdK}, then there exists a homomorphism
 \[
 \rho \colon K_2 \mathrm{Sp}(A) \rightarrow W(A)
 \]
 of abelian groups such that $\rho(\{x,y\}_{\alpha})=\llangle x,y \rrangle$ for all $x,y \in A^{\times}$
 \end{cor}
 
 \begin{proof}
 This is an immediate consequence of Theorem~\ref{thm:presentation_of_K2Sp} and Proposition~\ref{prop:relations_in_Witt_ring}.
 \end{proof}
 
 Note that the existence of an element $u_0 \in k \setminus k^2$ implies that $k$ is infinite (since all finite fields of characteristic $2$ are perfect). Thus $k[s,t] \slash (s^2,st,t^2)$ is a local ring with infinite residue field, hence $U$-irreducible (see \cite[Example~1.5.(d)]{VdK}). To show nontriviality of $\bar{z}=\{1+u^{-1} \bar{t},1+u\bar{s}\}_{\alpha} \{1+\bar{t},1+\bar{s}\}_{\alpha}^{-1}$, it suffices to check that its image under the homomorphism
 \[
 \rho \colon K_2 \mathrm{Sp}\bigl( k[s,t] \slash (s^2,st,t^2) \bigr) \rightarrow W\bigl( k[s,t] \slash (s^2,st,t^2) \bigr)
 \]
 is nontrivial. This is equivalent to the claim that
 \[
 \llangle 1+u^{-1} \bar{t}, 1+u\bar{s} \rrangle \quad \text{and} \quad \llangle 1+\bar{t}, 1+\bar{s} \rrangle
 \]
 are distinct elements of $W\bigl( k[s,t] \slash (s^2,st,t^2) \bigr)$.
 
 \begin{lemma}\label{lemma:image_of_zbar_in_WA}
 There exist isomorphisms
 \[
  \llangle 1+u^{-1} \bar{t}, 1+u\bar{s} \rrangle \cong A(u\bar{s},u^{-1}\bar{t}) \perp A(1,0) \quad \text{and} \quad \llangle 1+\bar{t}, 1+\bar{s} \rrangle \cong A(\bar{s},\bar{t}) \perp A(1,0)
 \]
 of symmetric bilinear spaces over $A=k[s,t] \slash (s^2,st,t^2)$. Thus we have
 \[
 \rho(\bar{z})=A(u \bar{s},u^{-1} \bar{t})-A(\bar{s},\bar{t})
 \]
 in $W(A)$.
 \end{lemma}
 
 \begin{proof}
 The second claim follows from the first, combined with the formula of $\bar{z}$ given in Lemma~\ref{lemma:final_form} and the defining property of the homomorphism $\rho$ of Corollary~\ref{cor:homomorphism_to_Witt_ring}.
 
 To see the first claim, note that we have
 \[
   \llangle 1+u^{-1} \bar{t}, 1+u\bar{s} \rrangle = \langle 1, 1+u^{-1} \bar{t}, 1+u\bar{s}, 1+u^{-1}\bar{t} + u\bar{s} \rangle
 \]
 by definition. The set
\[
v_1 = \begin{pmatrix}
0 \\ 1+u^{-1} \bar{t} + u \bar{s} \\ 0 \\ 1+u^{-1} \bar{t}
\end{pmatrix}, \quad
v_2 = \begin{pmatrix}
0 \\ 0 \\ 1+u^{-1} \bar{t} + u \bar{s} \\ 1+u \bar{s}
\end{pmatrix}, \quad
v_3 = \begin{pmatrix}
1 \\ 0 \\ 0 \\ 0
\end{pmatrix}, \quad
v_4 = \begin{pmatrix}
1 \\ 1 \\ 1 \\ 1
\end{pmatrix}
\] 
 forms a basis of $A^4$. To see this, it suffices to compute the determinant of the relevant matrix modulo the unique maximal ideal $(\bar{s},\bar{t})$ of $A$, which is straightforward.
 
 Using the fact that $(1+u\bar{s})^2=1$, hence $(1+u^{-1} \bar{t} + u \bar{s})(1+u\bar{s})=1+u^{-1}\bar{t}$ and similar computations one can check that the matrix of the diagonal bilinear form $\langle 1, 1+u^{-1} \bar{t}, 1+u\bar{s}, 1+u^{-1}\bar{t} + u\bar{s} \rangle$ with respect to this basis is given by
 \[
 \begin{pmatrix}
 u \bar{s} & 1 & 0 & 0 \\ 1 & u^{-1} \bar{t} & 0 & 0 \\ 0 & 0 & 1 & 1 \\ 0 & 0 & 1 & 0
 \end{pmatrix} \cong A(u\bar{s},u^{-1} \bar{t}) \perp A(1,0) \smash{\rlap{.}}
 \]
 Specializing this to $u=1$ we find that
 \[
\llangle 1+\bar{t}, 1+\bar{s} \rrangle \cong A(\bar{s},\bar{t}) \perp A(1,0)  
 \]
 holds.
 \end{proof}

 We have thus reduced the problem to showing that the two symmetric bilinear spaces
 \[
 A(\bar{s},\bar{t}) \quad \text{and} \quad A(u\bar{s}, u^{-1} \bar{t})
 \]
 represent distinct elements in the Witt ring $W \bigl(k[s,t] \slash (s^2,st,t^2)\bigr)$. We first show that they are not isomorphic if $u$ is not a square in $k$. Recall that the \emph{norm group} $\mathfrak{g}E$ of a symmetric bilinear space $E$ over $A$, with symmetric bilinear form $B$, is the subgroup of $(A,+)$ generated by the elements $B(v,v)$, $v \in E$.
 
 \begin{lemma}\label{lemma:not_isomorphic}
 The two bilinear spaces
 \[
 A(\bar{s},\bar{t}) \quad \text{and} \quad A(u\bar{s}, u^{-1} \bar{t})
 \]
 over $A=k[s,t] \slash (s^2,st,t^2)$ are not isomorphic if $u \in k \setminus k^{2}$.
 \end{lemma}
 
 \begin{proof}
 Clearly isomorphic bilinear spaces have equal norm groups, so we will show that the norm groups of these two spaces differ. We have for all $\alpha, \beta \in A$
\begin{align*}
\begin{pmatrix}
\alpha & \beta
\end{pmatrix}
\begin{pmatrix}
u\bar{s} & 1 \\ 1 & u^{-1}\bar{t}
\end{pmatrix}
\begin{pmatrix}
\alpha \\ \beta
\end{pmatrix} &=
\begin{pmatrix}
\alpha & \beta
\end{pmatrix}
\begin{pmatrix}
 \alpha u \bar{s} + \beta \\ \alpha + \beta u^{-1} \bar{t}
\end{pmatrix}\\
&= \alpha^2 u \bar{s}+\alpha \beta+\alpha \beta+\beta^2 u^{-1} \bar{t} \\
&=\alpha^2 u \bar{s}+\beta^2 u^{-1} \bar{t}
\end{align*}
since $k$ has characteristic $2$. Writing $\alpha=\alpha_0+\alpha_1 \bar{s} + \alpha_2 \bar{t}$ with $\alpha_i \in k$, we find that $\alpha^2=\alpha_0^2$. Similarly we have $\beta^2=\beta_0^2$ for some $\beta_0 \in k$. Since the Frobenius is a homomorphism, the set
\[
\{ \alpha_0^2 u \bar{s}+\beta_0^2 u^{-1} \bar{t} \; \vert \; \alpha_0, \beta_0 \in k \}
\]
is a subgroup of $(A,+)$, hence it is equal to the norm group $\mathfrak{g} A(u\bar{s},u^{-1} \bar{t})$. Specializing to $u=1$ we find that
\[
\mathfrak{g}A(\bar{s},\bar{t})= \{ \alpha_0^2 \bar{s}+\beta_0^2 \bar{t} \; \vert \; \alpha_0, \beta_0 \in k \}
\]
holds. Thus if $u$ is not a square, then $u\bar{s}$ lies in $\mathfrak{g}A(u\bar{s},u^{-1}\bar{t})$, but not in $\mathfrak{g}A(\bar{s},\bar{t})$.
 \end{proof}
 
 We can now apply the following theorem of Knebusch. Let $A$ be a local dyadic ring (meaning that $2$ is not a unit in $A$), with maximal ideal $\mathfrak{m}$. Let $\mathfrak{c}$ be the ideal generated by the squares $\mathfrak{m}^{(2)} \subseteq \mathfrak{m}$ and $2\mathfrak{m}$.
 
 \begin{thm}[Knebusch]\label{thm:Knebusch}
 Suppose that $\mathfrak{c}=0$. If $E$ and $F$ are two nondegenerate anisotropic bilinear spaces which represent equal elements in the Witt ring $W(A)$, then $E$ and $F$ are isomorphic.
 \end{thm}
 
 \begin{proof}
 This is \cite[Theorem~8.2.1]{KNEBUSCH}.
 \end{proof}

 For $A=k[s,t] \slash (s^2,st,t^2)$, we have $\mathfrak{c}=0$ since $\mathfrak{m}^2=0$ and $2=0$, so the above theorem is applicable in our case of interest. We can finally prove the main result of this section.
 
 \begin{thm}\label{thm:nontrivial_in_quotient}
 Let $u \in k \setminus k^{2}$. Then the image $\bar{z}$ of the element
 \[
 z=x_{\alpha}(us)x_{-\alpha}(u^{-1}t)x_{\alpha}\bigl((1+u)s\bigr)x_{-\alpha}(t) x_{\alpha}(s)x_{-\alpha} \bigl((1+u^{-1})t\bigr) \in \mathrm{StSp}(k[s,t])
 \]
 in $\mathrm{StSp}\bigl(k[s,t] \slash (s^2,st,t^2)\bigr)$ is nontrivial.
 \end{thm}
 
 \begin{proof}
 We have already reduced the claim to checking that
 \[
  A(\bar{s},\bar{t}) \quad \text{and} \quad A(u\bar{s}, u^{-1} \bar{t})
 \]
 represent distinct elements in $W\bigl(k[s,t]\slash (s^2,st,t^2)\bigr)$, so suppose that they are equal in the Witt ring. Since the determinants of the two defining matrices are both equal to $1$, these spaces are nondegenerate. In order to apply Knebusch's theorem, we need to check that these spaces are anisotropic, that is, the bilinear form does not vanish on any direct summand. Since $A=k[s,t] \slash (s^2,st,t^2)$ is local, each proper direct summand is given by a split monomorphism
 \[
 \begin{pmatrix}
 a \\ b
\end{pmatrix} \colon  A \rightarrow A^2
 \]
 and the locality of $A$ implies that one of $a,b$ is a unit. From the computation at the beginning of the proof of Lemma~\ref{lemma:not_isomorphic} it follows that the bilinear form does not vanish on this vector, so these spaces are indeed anisotropic. From Theorem~\ref{thm:Knebusch} it follows that the two spaces are isomorphic. But this contradicts Lemma~\ref{lemma:not_isomorphic}. Thus $  A(\bar{s},\bar{t})$ and $A(u\bar{s}, u^{-1} \bar{t})$ represent different elements in the Witt ring.
 
  From Lemma~\ref{lemma:image_of_zbar_in_WA} we know that
 \[
 \rho(\bar{z})=A(u\bar{s}, u^{-1} \bar{t})-A(\bar{s},\bar{t}) \in W(A) \smash{\rlap{,}}
 \]
 so $\rho(\bar{z})$ is a nonzero element of the Witt ring. It follows that $\bar{z} \in K_2 \mathrm{Sp}(A) \subseteq \mathrm{StSp}(A)$ is nontrivial.
\end{proof}

\section{Counterexamples}\label{section:counterexamples}

 We again fix a field $k$ of characteristic $2$ and we let $R \defl k[a,x,y] \slash a^2+xy$. Recall that for every $u \in k$, we have the matrix
 \[
 M(u)=\begin{pmatrix}
 1+ (1+u)a^2 & (1+u)(1+a)y \\ (1+u^{-1})(1+a)x & 1+(1+u^{-1})a^2
 \end{pmatrix}
 \begin{pmatrix}
 1 & 0 \\ (1+u^{-1})x & 1
 \end{pmatrix}
 \]
 in $\mathrm{SL}_2(R)$. We can now prove the following theorem.
 
 \begin{thm}\label{thm:matrix_nontrivial_in_K1Sp}
 If $u$ is not a square, then $[M(u)] \in K_1 \mathrm{Sp}(R)$ is nontrivial.
 \end{thm}

\begin{proof}
 In order to apply Proposition~\ref{prop:gubeladze}, we need to show that homotopy invariance holds and that the image $\bar{z}$ of $z \in \mathrm{StSp}(k[s,t])$ is nontrivial in the symplectic Steinberg group $\mathrm{StSp}\bigl(k[s,t] \slash (s^2,st,t^2)\bigr)$. The first claim follows from Corollary~\ref{cor:homotopy_invariance} and the second claim is established in Theorem~\ref{thm:nontrivial_in_quotient}.
\end{proof} 

 Using this, we can show that there are counterexamples to the Hermite ring conjecture.
 
 \begin{prop}\label{prop:nonfree_stably_free_example}
 Let $R^{\prime}=k[a,x,y,t] \slash \bigl( t^2+t(a^2+xy) \bigr)$. If $k$ is not perfect, then there exists a projective $R^{\prime}$-module $P$ such that $P \oplus R^{\prime} \cong (R^{\prime})^3$, but $P$ itself is not free.
 \end{prop}

 \begin{proof}
  The ring $R^{\prime}$ fits in the Milnor square
  \[
  \xymatrix{R^{\prime} \ar[r] \ar[d] & k[a,x,y] \ar[d]^{\pi} \\ k[a,x,y] \ar[r]_-{\pi} & k[a,x,y] \slash (a^2+xy)=R }
  \]
  where $\pi$ denotes the canonical projection. Fix a $u \in k$ which is not a square. The matrix $M(u) \in \mathrm{SL}_2(R)$ defines by Milnor patching a projective $R^{\prime}$-module $P$. From the fact that $\bigl(\begin{smallmatrix} M(u) & 0\\ 0 & 1 \end{smallmatrix} \bigr)$ is elementary (see Lemma~\ref{lemma:matrix_1_stably_elementary}) we conclude that $P \oplus R^{\prime}$ is free.
  
 To see that $P$ is not free, we need to show that $M(u)$ does not lie in the image of $\pi$. If it were, then we would have $[M(u)]=1 \in K_1 \mathrm{Sp}(R)$ since $K_1 \mathrm{Sp}(k[a,x,y])=1$ (see \cite[Theorem~3.14]{KOPEIKO}). This contradicts the nontriviality of $[M(u)]$ established in Theorem~\ref{thm:matrix_nontrivial_in_K1Sp}, so $P$ is indeed not free.
\end{proof}  
 
 The ring $R^{\prime}$ is graded, with $\lvert t \rvert=2$ and $\lvert a \rvert=\lvert x \rvert=\lvert y \rvert=1$. Thus there exists a ring homomorphism $h \colon R^{\prime} \rightarrow R^{\prime}[s]$ with $\mathrm{ev}_0 h=R^{\prime} \rightarrow k \rightarrow R^{\prime}$ and $\mathrm{ev}_1 h=\id$.
 
\begin{proof}[Proof of Theorem~\ref{thm:counterexample_to_Swan_question}]
If we fix a projective module $P$ as in Proposition~\ref{prop:nonfree_stably_free_example} above, it follows that $h_{\ast} P$ is a projective $R^{\prime}[s]$-module which is not extended from $R^{\prime}$. Indeed, otherwise $P \cong (\mathrm{ev}_1 h)_{\ast} P$ would be isomorphic to $(\mathrm{ev}_0 h)_{\ast} P$, hence free. 
\end{proof} 

 This module $h_{\ast } P$ thus answers Swan's Question (B) in \cite[p.~114]{SWAN} in the negative.
 
 \begin{cor}
 If $k$ is not perfect, then at least one of the rings 
 \[
 \{R^{\prime}_{\mathfrak{m}}[s] \; \vert \; \mathfrak{m} \subseteq R^{\prime} \; \text{maximal} \}
 \]
 is not a Hermite ring.
 \end{cor}
 
 \begin{proof}
 If all these rings were Hermite, then each module $(h_{\ast} P)_{\mathfrak{m}}$ would be free, hence extended from $R^{\prime}_{\mathfrak{m}}$. Quillen's local-global principle would therefore imply that $h_{\ast} P$ is extended from $R^{\prime}$, and we have noted above that this is not the case for $P$ as in Proposition~\ref{prop:nonfree_stably_free_example}.
 \end{proof}

 Since all local rings are Hermite, the above corollary shows that there are counterexamples to the Hermite ring conjecture. The following corollary shows that there are also counterexamples among rings with certain nice properties. 
 
 \begin{cor}
 There exist local unique factorization domains $A$ (which are also local rings of global complete intersections in a polynomial ring over $\mathbb{Z}$) such that $A[t]$ is not Hermite.
 \end{cor}
 
 \begin{proof}
 We know from \cite[Theorem~1.4]{SCHAEPPI_LAURENT} that the Hermite ring conjecture would hold if it holds for all rings $A$ as in the statement.
 \end{proof}
 
\bibliographystyle{amsalpha}
\bibliography{hermite}

\end{document}